\documentclass[a4paper]{amsart}
\usepackage[T1]{fontenc}
\usepackage{lmodern}
\usepackage{amssymb}
\usepackage[all]{xy}
\usepackage[inline,shortlabels]{enumitem}
\usepackage{nicefrac,stmaryrd}
\usepackage{graphicx, import}
\usepackage{mathtools}
\usepackage[british]{babel}
\usepackage{microtype}

\usepackage{tikz-cd, tikz}
\usepackage[pdftitle={Actions of fundamental groupoid},
pdfauthor={Holkar, Hossain, Kulkarni},
  pdfsubject={} 
]{hyperref}
\usepackage[lite]{amsrefs}
 \usepackage{amsthm,amsfonts}
 \usepackage{nicefrac,color}

\numberwithin{equation}{section}
\theoremstyle{plain}
\newtheorem{theorem}[equation]{Theorem}
\newtheorem*{theorem*}{Theorem}
\newtheorem{lemma}[equation]{Lemma}
\newtheorem*{lemma*}{Lemma}
\newtheorem{lemdef}[equation]{Lemma and definition}
\newtheorem{proposition}[equation]{Proposition}
\newtheorem*{proposition*}{Proposition}

\newtheorem{corollaries}[equation]{Corollaries}

\theoremstyle{definition}
\newtheorem{definition}[equation]{Definition}
\newtheorem{observation}[equation]{Observation}

\theoremstyle{remark}
\newtheorem{remark}[equation]{Remark}
\newtheorem{example}[equation]{Example}

\newcommand{\CatAct}{\mathcal{A}}

\newcommand{\CatEAct}{\mathcal{E}}
\newcommand{\CatCovSp}{\mathcal{COV}}
\newcommand*{\defeq}{\mathrel{\vcentcolon=}}
\newcommand{\N}{\mathbb N}
\newcommand{\R}{\mathbb R}
\newcommand*{\base}[1][G]{{#1}^{(0)}}
\newcommand{\nb}{\nobreakdash}

\newcommand{\inverse}{^{-1}}
\newcommand{\homeo}{\approx}
\newcommand{\iso}{\simeq}

\newcommand{\FGd}{\mathrm{\Pi}_1}
\newcommand{\FGp}{\mathrm{\pi}_1}
\newcommand{\UInt}{\mathbb{I}}
\newcommand{\Cov}[1][c]{\mathrm{c}}
\newcommand{\Lift}[1][\gamma]{\tilde #1}
\newcommand{\PathS}[1]{\mathrm{P}\!#1}
\newcommand{\etale}{{\'e}tale}
\newcommand*{\e}{\mathrm e}
\newcommand*{\Cst}{\textup C^*}

\title[Actions of fundamental groupoid]{Topological Fundamental
  Groupoid. II. An action category of the fundamental groupoid}

\author{Rohit Dilip Holkar} \email{rohit.d.holkar@gamil.com}

\author{Md Amir Hossain}
\email{mdamir18@iiserb.ac.in}

\author{Dheeraj Kulkarni}
\email{dheeraj@iiserb.ac.in}

\address{Department
	of Mathematics, Indian Institute of Science Education and Research
	Bhopal, Bhopal Bypass Road, Bhauri, Bhopal 462 066, Madhya Pradesh,
	India.}

       \thanks{\emph{Subject class.} 14H30,
        22A22, 57S, 37C85}

\begin{document}
\maketitle{}

\begin{abstract}
  For a path connected, locally path connected and semilocally simply
  connected space~\(X\), let~\(\Pi_1(X)\) denote its topologised fundamental
  groupoid as established in the first article of this
  series. Let~\(\mathcal{E}\) be the category
  of~\(\Pi_1(X)\)-spaces in which the momentum maps are local
  homeomorphisms. We show that this category is isomorphic to that of
  covering spaces of~\(X\). Using this, we give different
  characterisations for free or proper actions of the fundamental
  groupoid in~\(\mathcal{E}\). 
\end{abstract}

\tableofcontents{}

\section*{Introduction}
\label{sec:introduction}

In the earlier article~\cite{Holkar-Hossain2023Top-FGd-I}, we
topologise the fundamental groupoid of locally path connected and
semilocally simply connected spaces in a natural way. We discussed the
interrelationship between the topology of the underlying space and the
that of the fundamental groupoid in detail. In current article, we
turn our attention to the action category of the fundamental
groupoids.

The covering spaces carry a natural action of the fundamental
groupoid, namely, by evaluation at a lifted path (Proposition~\ref{prop:nature-of-act-of-FG}).  This has been a
standard observation,
e.g.\cite{Brown-Danesh-1975-Top-FG-1},~\cite{Brown-Danesh-Naruie-1976-Top-FG-2},\cite{Brown2006Topology-book}
and~\cite{Reinhart1983Folliations-book}. Relation of this action in
constructing the covering spaces has been a central attraction in
above literature. We take a different approach and wish to study these
actions as actions of a locally compact groupoid on spaces.

We are
also interested in finding out when these actions are free and
proper. Our interests are motivated by the intension of studying the
\(\Cst\)\nb-correspondences
(\cite{Muhly-Tomforde-2005-Topological-quivers},\cite{Holkar2017Construction-of-Corr},\cite{Tu2004NonHausdorff-gpd-proper-actions-and-K})
associated with \(\FGd(X)\)\nb-spaces. For constructing these
\(\Cst\)\nb-correspondences, we need to construct the Haar system
on~\(\FGd(X)\) and understand the free and proper
\(\FGd(X)\)\nb-spaces. We study the \(\FGd(X)\)\nb-action in current
and Haar systems in a followup article.

We notice that the covering spaces constitute an interesting category
of actions of a fundamental groupoid.
Theorem~\ref{thm:classify-cover-space} characterises the covering
spaces as
\(\FGd(X)\)\nb-spaces. Theorem~\ref{thm:iso-act-and-cov-categiries}
describe the category of covering spaces as a category of
certain~\(\FGd(X)\)\nb-spaces. We describe free
(Proposition~\ref{prop:condi-for-free-action}) and proper
(Theorem~\ref{thm:suff-cond-proper-action}) actions \emph{nicely} in
this category. While investigating the proper actions, we also
describe \emph{small} compact sets in~\(\FGd(X)\) (discussion
following Proposition~\ref{prop:kinetics-is-covering} and
Equation~\ref{eq:inverse-im-of-small-cpt}); this description could be
useful for practical or computation purposes.

\medskip

\paragraph{\emph{Organisation of the article:}} In the
Section~\ref{sec:prilim-TFGd-2}, we discuss preliminaries: topological
groupoids~\S~\ref{sec:topol-group} and their
actions~\S~\ref{sec:actions-groupoids}. In this section, we also
recall some required facts from preceding article of this series.

In Section~\ref{sec:acti-categ-fund}, we prove the first main result
Theorems~\ref{thm:classify-cover-space}
and~\ref{thm:iso-act-and-cov-categiries} which says that the category
of covering space over a space can be identified with the category of
\(\FGd(X)\)\nb-spaces in which the momentum map is a local
homeomorphism. We characterise the free actions
(Proposition~\ref{prop:condi-for-free-action}).

In the last section, first we prove that the map of \emph{kinetics} of
an action (Equation~\eqref{eq:kine}) is a local homeomorphism,
Proposition~\ref{prop:kinetics-is-covering}; this observation is next
used to characterise the proper actions
(Theorem~\ref{thm:suff-cond-proper-action}) of the fundamental
groupoid.

\section{Preliminaries}
\label{sec:prilim-TFGd-2}

\subsection{Topological groupoids}
\label{sec:topol-group}

In this second article, we continue to follow the conventions
established in the first article~\cite{Holkar-Hossain2023Top-FGd-I}.
Nonetheless, here is a quick recap of notation about groupoids: for
us, a \emph{groupoid}~\(G\) is a small category in which every arrow
is invertible. We abuse the notation and consider the set of
units~\(\base\) a subset of~\(G\). It is standard exercise that for an
element \(\gamma \in G\), \(\gamma \inverse \gamma = s(\gamma)\) and
\(\gamma \gamma \inverse = r(\gamma)\) are the range and source
of~\(\gamma\). The fibre product
\(G\times_{s, \base, r} G = \{(\gamma, \eta) \in G\times G : s(\gamma)
=r(\eta)\}\) is called the set of all composable pairs of \(G\) and
it's denoted by \(G^{(2)}\).

The groupoid~\(G\) is called \emph{topological} if it carries a
topology in which the source map~\(s\colon G\to \base\), range
map~\(r\colon G\to \base\), the inversion
map~\(\mathrm{inv}\colon G\to G\) and the
multiplication~\(m\colon G^{(2)} \to G\) are continuous; here the
space of units is given the subspace topology, and
\(G^{(2)}\subseteq G\times G\) carries the subspace topology.

The topological groupoid~\(G\) is called locally compact (or Hausdorff
or second countable) if the topology is locally compact (respectively,
Hausdorff or second countable) plus the space of units is
Hausdorff. For us locally compact spaces are not necessarily
Hausdorff, see~\cite{Holkar-Hossain2023Top-FGd-I}. However, unless the
reader is bothered about non-Hausdorff case, they may simply consider
locally compact as locally compact Hausdorff. Second countable or
paracompact spaces are assumed to be Hausdorff. We refer the reader to
Renault's book~\cite{Renault1980Gpd-Cst-Alg} and Tu's
article~\cite{Tu2004NonHausdorff-gpd-proper-actions-and-K} for basics
of locally compact groupoids and their actions.

Given units \(x, y\in \base\), we define the following closed subspace
of~\(G\):
\[
  G^x \defeq r^{-1}(x),\quad G_y \defeq s^{-1}(y)\quad \text{ and }
  \quad G^x_y \defeq G^x\cap G_y.
\]
In fact, \(G^x_x\) is a topological group called the \emph{isotropy
  group} at~\(x\). In general, for sets~\(A,B\subseteq \base\), we
define
\[
  G^A \defeq r^{-1}(A),\quad G_B \defeq s^{-1}(B)\quad \text{ and }
  \quad G^A_B \defeq G^A\cap G_B.
\]

\begin{observation}
  \label{obs:cardinality-of-iso-and-edges}
  Suppose that~\(\gamma\) is an arrow in a groupoid~\(G\); write
  \(x=s(\gamma)\) and \(y=r(\gamma)\). Then the cardinality of the
  isotropy at~\(x\) and \(G_x^y\) is same. This is because the
  function
  \[
    \phi\colon G_x^x\to G_x^y,\quad \eta\mapsto \gamma\eta \quad
    \text{for }\eta\in G_x^x
  \]
  is a bijection. The inverse of the function is given by
  \[
    \phi\inverse\colon G_x^y\to G_x^x,\quad \eta'\mapsto
    \gamma\inverse \eta' \quad \text{for }\eta'\in G_x^y.
  \]
\end{observation}

Next we discuss the fundamental groupoid. Consider a space~\(X\). For
a set~\(U\subseteq X\), we write~\(\PathS{U}\) for the set of all
paths in~\(U\). For a path~\(\gamma \in \PathS{X}\), \(\gamma(0)\) is
called the initial or starting point and~\(\gamma(1)\) the terminal or
end point of~\(\gamma\). For~\(\gamma\) as before, \(\gamma^-\)
denotes the path \emph{opposite} to~\(\gamma\). The concatenation of
paths is denoted by~\(\oblong\). The fundamental groupoid of~\(X\) is
denoted by~\(\FGd(X)\). The fundamental group of~\(X\) at~\(x\in X\)
is denoted by~\(\FGp(X,x)\). We shall use the fact that~\(\FGd(X)\) is
the quotient of~\(\PathS{X}\) by the equivalence relation of endpoint
fixing path homotopy.  If \(X\) path connected, we simply
write~\(\FGp(X)\) instead of~\(\FGp(X,x)\).  Finally, since we direct
the arrows in a groupoid from \emph{right to left}, we shall think
that a path starts from right and ends on left in oppose to the
standard convention.
      
      \begin{example}[The fundamental groupoid]\label{exa:fund-gpd}
        Let \(X\) be a locally path connected and semilocally simply
        connected space. Then, we prove in the preceding
        article~\cite{Holkar-Hossain2023Top-FGd-I}, that the
        fundamental groupoid~\(\FGd(X)\) of~\(X\) can be equipped with
        a topology so that it becomes a topological groupoid. Equipe
        the set of all paths, \(\PathS{X}\), in~\(X\) with the
        compact-open topology. Then the quotient topology
        on~\(\FGd(X)\) induced by the compact-open topology make the
        fundamental groupoid topological~\cite[Theorem
        2.8]{Holkar-Hossain2023Top-FGd-I}. We call this quotient
        topology the \emph{CO' topology}.

        There is another natural way to topologise~\(\FGd(X)\) as
        follows. For path connected and relatively inessential open
        sets~\(U,V\in X\) and a path~\(\gamma\) in~\(X\) starting at a
        point in~\(V\) and ending at a point in~\(U\), define for the
        following subset of~\(\FGd(X)\)
        \[
          N([\gamma], U,V) \defeq \{[\delta \oblong \gamma \oblong
          \omega] : \delta\in \PathS{U} \text{ with } \delta(0) =
          \gamma(1), \text{ and } \omega\in \PathS{V} \text{ with
          }\omega(1) = \gamma(0)\}.
        \]
        Then the sets of above form a basis for a topology
        on~\(\FGd(X)\) which we call \emph{the UC
          topology}. Proposition~2.4
        in~\cite{Holkar-Hossain2023Top-FGd-I} shows that for a locally
        path connected and semilocally simply connected space~\(X\),
        the UC and CO' topologies on the fundamental groupoid are
        same. The fundamental groupoid is not an \etale\ groupoid but
        a \emph{locally trivial} (\cite[Definition
        1.4]{Holkar-Hossain2023Top-FGd-I}) one.

        In the groupoid~\(\FGd(X)\), the range and source maps (which
        are basically the evaluations at~\(1\) and \(0\in [0,1]\),
        respectively) are open~\cite[Corollay
        2.7]{Holkar-Hossain2023Top-FGd-I}. The space of
        units~\(\base[\FGd(X)]\) consists of constant paths and can be
        identified with~\(X\)~\cite[Corollary
        2.9(1)]{Holkar-Hossain2023Top-FGd-I}. Assume that~\(X\) is
        also path connected. Then, for a unit~\(x\in X\), the
        fibres~\(\FGd(X)_x\) or~\(\FGd(X)^x\) can be identified with
        the simply connected covering space of~\(X\) (constructed
        using either the paths starting at~\(x\) or ending
        at~\(x\))~\cite[Corollary~2.9(2)]{Holkar-Hossain2023Top-FGd-I}. Moreover,
        the isotropy group \(\FGd(X)_x^x\) at~\(x\) is basically the
        fundamental group of~\(X\), and it is discrete
        ~\cite[Corollary~2.9(3)]{Holkar-Hossain2023Top-FGd-I}.

        The fundamental groupoid is Hausdorff (or locally compact or
        second countable) \emph{iff} the underlying spaces is
        so~\cite[Section 3]{Holkar-Hossain2023Top-FGd-I}.
      \end{example}

      \begin{example}[Fundamental groupoid of a group]
        \label{exa:fgd-of-gp}
        In earlier Example~\ref{exa:fund-gpd}, additionally assume
        that~\(X\) a topological group. Let~\(H\) be its covering
        group with the homomorphism~\(p\colon H\to X\) as the covering
        map. Then~\(H\) acts on~\(X\) through~\(p\). Theorem~2.21
        in~\cite{Holkar-Hossain2023Top-FGd-I} prove that the
        topological fundamental groupoid~\(\FGd(X)\) is isomorphic to
        the transformation groupoid~\(H \ltimes X\) of above action
        of~\(H\) on~\(X\).
      \end{example}

      \subsection{Actions of groupoids}
      \label{sec:actions-groupoids}

\begin{definition}
  Let \(G\) be a locally compact Hausdorff groupoid, and let \(X\) be
  a topological space with a continuous momentum map
  \(r_X\colon X \to \base\). We call \(X\) is a left \(G\)\nb-space
  (or \(G\) act on \(X\) from left ) if there is a continuous map
  \( \sigma \colon G \times_{s, \base, r_X} X \to X\) satisfying the
  following conditions:
  \begin{enumerate}
  \item \( \sigma (r_X(x), x) = x\) for all \(x\in X\);
  \item if \((\gamma, \eta) \in G^{(2)}\) and
    \((\eta, x) \in G \times_{s, \base, r_X} X \), then
    \((\gamma\eta, x), (\gamma, \sigma(\eta, x) ) \in G \times_{s,
      \base, r_X} X \) and
    \(\sigma ((\gamma \eta) , x) = \sigma (\gamma ,\sigma(\eta, x))\).
  \end{enumerate}
\end{definition}
      
We shall abuse the natation \(\sigma (\gamma, x)\) and simply write
\(\gamma \cdot x\) or \(\gamma x\). We shall often say `\(X\) is a
left (or right) \(G\)\nb-space' for a groupoid~\(G\); here it will be
\emph{tacitly assumed} that~\(r_X\) (respectively, left) is the
momentum map. The range (or source) map is the momentum map for the
left (respectively, right) multiplication action of a groupoid on
itself.

A groupoid~\(G\) acts on its space of units, from left, as follows:
for \(\gamma\in G\) and \(x\in G\), the action is defined if
\(s(\gamma) = x\) and is given by~\(\gamma x = r(\gamma)\). The
identity map on \(\base\) is the momentum map for this
action. Similarly a right action of~\(G\) on~\(\base\) is defined.

Given \(G\)\nb-spaces \(X\) and \(Y\), by an \emph{equivariant map} we
mean a function \(f\colon X\to Y\) such that \(r_Y\circ f = r_X\) and
\(f(\gamma x) = \gamma f(x)\) for all composable pairs
\((\gamma,x)\in G\times_{s, \base, r_X} X\).

\begin{example}[Transformation groupoid]
  \label{exa:transormation-gpd}
  For a continuous (right) action of a groupoid~\(G\) on a
  space~\(X\), one can construct the transformation groupoid which is
  denoted by~\(X\rtimes G\). The underlying space of the groupoid is
  the fibre product~\(X\times_{s_X, \base[G], r} G\); two elements
  \((x,g)\) and \((y,t)\) in~\(X\rtimes G\) are composable iff
  \(y= x\cdot g\), and the composition is given by
  \((x,g)(y,t) \defeq (x, gt)\); the inverse of~\((x,g)\) is given by
  \((x,g)^{-1} \defeq (x\cdot g, g^{-1})\). 
  For a left \(G\)\nb-space \(Y\), the transformation
  groupoid is defined similarly and is denoted by~\(G\ltimes Y\).
\end{example}

Next is a characterisation of spaces on which a transformation
groupoid can act.

\begin{lemma}[Lemma~2.7
  in~\cite{Emerson-Meyer2010Dualities-in-Equi-Kasparov-theory}]
  \label{lem:homo-of-trans-gpds}
  Let \(G\ltimes X\) be a transformation groupoid for an action of a
  groupoid~\(G\) on a space~\(X\). Then \(G\ltimes X\) acts on a space
  \(Y\) with \(\rho\colon Y\to X\) as momentum map \emph{iff} \(\rho\)
  is a \(G\)\nb-equivariant map of spaces. Thus there is a one-to-one
  correspondence between \(G\)\nb-equivariant maps
  \(\rho\colon Y\to X\) and \(G\ltimes X\)\nb-spaces~\(Y\).
\end{lemma}

Assume that \(X\) is a \(G\)\nb-space for a groupoid~\(G\). While
studying the groupoid actions, the following map
\begin{equation}\label{eq:kine}
  a\colon G\times_{s, \base[G], r_X} X \to X\times X, \quad a\colon
  (\gamma, x)  \mapsto (\gamma x, x)
\end{equation}
turns out useful. Although, this map does not have a standard name,
for the current article we call it \emph{the kinetics}\footnote{A
  better name is welcome!} or \emph{the map of the kinetics} of the
action. Observation~2.10 in~\cite{Holkar-Hossain2023Top-FGd-I} shows
that the map of kinetics for the action of~\(\FGd(X)\) on its space of
units is a local homeomorphism where~\(X\) is a locally path connected
and semilocally simply connected space.
      
\begin{lemdef}
  \label{lem:free-act}
  Let \(G\) be a groupoid acting on a space~\(X\).  Then the following
  statement are equivalent:
  \begin{enumerate}
  \item the map of kinetics of the action is one-to-one.
  \item For every~\(x\in X\), the stabiliser~\((G\ltimes X)^x_x\) is
    the trivial group.
  \end{enumerate}
  If any of the above condition holds, we call the \(G\)\nb-action
  on~\(X\) \emph{free}.
\end{lemdef}

\noindent Being a standard fact, we leave the proof of above lemma to
reader.

A map of space \(f\colon X\to Y\) is called proper if \(f\inverse(K)\)
is compact if~\(K\subseteq Y\) is a compact set.
      
      \begin{lemdef}
        Let~\(G\) be a locally compact Hausdorff groupoid acting on a
        locally compact Hausdorff space~\(X\). Then the following
        statements are equivalent:
	\begin{enumerate}
        \item the maps of the kinetics of the action is proper;
        \item for any pair of compact subsets \(T, S\) of~\(X\), the
          set \(\{\gamma \in G : \gamma T \cap S \neq \emptyset\}\) is
          a compact set of \(G\).
	\end{enumerate}
        The action of~\(G\) on~\(X\) is called \emph{proper} if any of
        the above condition holds. And then the transformation
        groupoid~\(G\ltimes X\) is called proper.
      \end{lemdef}

      \noindent The proof of last lemma is also standard,
      e.g. see~\cite[Proposition~2.17]{Williams2019A-Toolkit-Gpd-algebra}. A
      groupoid~\(G\) is called proper if its action on the space of
      units is proper.

      \begin{observation}
	\label{obs:isotropy-of-proper-action}
	If a groupoid \(G\) acts properly on a space~\(X\), then the
        isotropy at any point~\(x\in X\) is a compact; for it is
        inverse image of~\((x,x)\) under the map of kinetics of the
        action.
      \end{observation}

      \section{And action category of a fundamental groupoid and free
        actions}
      \label{sec:acti-categ-fund}

      \begin{definition}
	\label{def:path-lift-prop}
	Let \(X\) and \(Y\) be spaces, and \(f\colon Y\to X\) a
        continuous surjection. We say that \(f\) has
	\begin{enumerate}
        \item the path lifting property (or \emph{the unique path
            lifting property}) if for any path
          \(\gamma\colon \UInt\to X\) and a point
          \(y\in f\inverse(\gamma(0))\), there is a path
          (respectively, a unique path) \(\tilde{f}\colon\UInt \to Y\)
          starting at \(y\) and \(f\circ \tilde{f} = \gamma\).
        \item \emph{the homotopy lifting property} (or \emph{the
            unique homotopy lifting property} ) if \(f\) has path
          lifting property (respectively, the unique path lifting
          property) and for given two paths
          \(\gamma, \alpha \colon \UInt\to X\) with
          \(\gamma(0)=\alpha(0)\) and \(\gamma(1)=\alpha(1)\); an
          endpoint fixing homotopy
          \(\Gamma\colon \UInt\times\UInt\to X\) of~\(\gamma\)
          with~\(\alpha\); and a point \(y\in f\inverse(\gamma(0))\),
          there is a function (respectively, a unique function)
          \(\tilde{\Gamma}\colon \UInt\times\UInt\to Y\) with the
          properties that \(\tilde{\Gamma}|_{\{0\}\times\UInt}\) is a
          lift (respectively, the unique lift) \(\tilde{\gamma}\)
          of~\(\gamma\) starting
          at~\(y\in Y\);~\(\tilde{\Gamma}|_{\{1\}\times\UInt}\) is a
          lift (respectively, the unique lift)~\(\tilde{\alpha}\) of
          \(\alpha\) starting at \(y\in Y\); and
          \(f \circ \tilde{\Gamma} = \Gamma\).
	\end{enumerate}
      \end{definition}

      It is a standard fact that covering maps have the unique path
      lifting and unique homotopy lifting properties. Proposition~3 of
      Chapter--5-6A in~\cite{DoCarmo1976Diff-Geo-Curves-Surface} says
      that a local homeomorphism having the unique path lifting
      property also has unique homotopy lifting property.

      \begin{remark}[Functoriality of the unique path lifting
        property]\label{rem:img-of-lift-of-path}
  Suppose \(Y_1\xrightarrow{p_1} X \xleftarrow{p_2} Y_2\) are two
  mappings which have unique path lifting properties. Assume that
  \(f\colon Y_1 \to Y_2\) is a continuous map such that
  \(p_1 = p_2 \circ f\). For a given path \(\gamma\) in~\(X\), choose
  \(y_1\in Y_1\) with \(p_1(y_1) = \gamma(0)\). Let \(y_2=
  f(y_1)\). If \(\tilde{\gamma}_{y_1}\) is the unique lift
  of~\(\gamma\) in~\(Y_1\) starting at~\(y_1\), then
  \(f\circ \tilde{\gamma}_{y_1}\) is the unique lift of~\(\gamma\)
  in~\(Y_2\) starting at~\(y_2\) as~\(p_1 = p_2 \circ f\).
\end{remark}

\subsection{Covering spaces as \(\FGd(X)\)-spaces}
\label{sec:act-of-FG}

Let~\(X\) be a locally path connected and semilocally simply connected
space, and \(\Cov\colon Y\to X\) a covering map. For a path
\(\gamma\in \PathS{X}\) and \(y\in \Cov\inverse(\gamma(0))\),
by~\(\tilde \gamma_{y}\) we shall denote the unique lift of~\(\gamma\)
starting at~\(y\). As~\(\Cov\) also has the unique homotopy lifting
property, each pair \(([\gamma], y)\) where \(\gamma\) is a path
in~\(X\) and \(y\in \Cov\inverse(\gamma(0))\), determines the unique
element~\([\tilde \gamma_{y}]\in\FGd(Y)\). Using this observation, we
define a (left) action of \(\FGd(X)\) on~\(Y\) as follows:
\begin{enumerate}[(i),leftmargin=*]
\item \(\Cov\) is the momentum map for the action;
\item For each pair~\(([\gamma], y)\) in the fibre
  product~\(\FGd(X)\times_{s, X, \Cov} Y\), the action
  \([\gamma]y \defeq \tilde\gamma_y(1)\).
\end{enumerate}

\noindent For the sake of clarity, the fibre
product~\(\FGd(X)\times_{s, X, \Cov} Y = \{([\gamma], y)\in
\FGd(X)\times Y : \gamma(0) = \Cov(y)\}\).

We shall refer this action of the fundamental groupoid~\(\FGd(X)\)
on~\(Y\) as \emph{the} (left) action of the groupoid on the covering
space. A right action can be defined similarly.

Notice that the last action can be defined, in general, for any
mapping having the unique path and homotopy lifting properties.
\medskip

For a covering map \(\Cov\colon Y\to X\), an evenly covered 
neighbourhood of a point has the standard meaning as in
Hatcher~\cite{Hatcher2002Alg-Top-Book}. Consider an evenly covered
open set \(U\subseteq X\); write
\(\Cov\inverse(U) = \sqcup_{\alpha} \tilde{U}_\alpha\) where each
\(\tilde{U}_\alpha\) is homeomorphic to~\(U\) via~\(\Cov\). We call
each~\(\tilde{U}_\alpha\) a slice over~\(U\).  Since evenly covered
neighbourhoods form a basis for the topology of~\(X\), the slices also
form a basis for the topology of~\(Y\).

\begin{proposition}
  \label{prop:nature-of-act-of-FG}
  Let \(X\) be a locally path connected and semilocally simply
  connected space and \(\Cov\colon Y\to X\) a covering map.
  \begin{enumerate}
  \item The action of \(\FGd(X)\) on \(Y\) is continuous.
  \item \label{cond:stab} The stabilizer of \(y\in Y\) is the subgroup
    \(\Cov_*(\FGp(Y,y))\iso \FGp(Y,y)\) of the fundamental
    group~\(\FGp(X,\Cov(y))\); here \(\Cov_*\) is the homomorphisms of
    fundamental group(oid)s that~\(\Cov\) induces.
  \item Assume \(Y\) is also path connected. Then the action
    of~\(\FGd(X)\) on~\(Y\) is free if and only if \(Y\) is simply
    connected.
  \end{enumerate}
\end{proposition}

\begin{proof}
  \noindent (1): The momentum map \(\Cov\) is continuous. So we only
  need to show that the map
  \[
    \sigma \colon \FGd(X)\times_{s,X,\Cov} Y \to Y, \quad \sigma
    \colon ([\gamma],y)\mapsto [\gamma]y \defeq \tilde{\gamma_y}(1)
  \]
  is continuous. Let \(W\subseteq Y\) be a given open set. For given a
  point \(([\gamma], y)\in \sigma \inverse(W)\), we construct a basic
  open neighbourhood~\(Q\) of~\(([\gamma], y)\) such
  that~\(Q \subseteq \sigma \inverse(W)\) to prove the
  continuity~\(\sigma\).
	
  Let~\(y'\) denote \(\tilde{\gamma}_y(1)\). Thus,
  \(\tilde{\gamma}_y\) starts at~\(y\) and ends at~\(y'\in W\). Choose
  path connected relatively inessential slices~\(V\) and \(U\) over
  some evenly covered neighbourhoods of~\(\Cov(y')\) and~\(\Cov(y)\),
  respectively. Additionally, as slices form a basis for the topology
  of~\(Y\), we can choose~\(V\subseteq W\).  Then
  \[
    Q \defeq \bigl(N([\gamma],\Cov(V), \Cov(U)) \times U \bigr) \cap
    \bigl(\FGd(X)\times_{s,X,\Cov} Y\bigr)
  \]
  is a nonempty basic open neighbourhood of~\(([\gamma], y)\)
  in~\(\FGd(X)\times_{s,X,\Cov} Y\), and
  clearly~\(\sigma (Q) = V\subseteq W\).
	
  \noindent (2): Here we basically want to describe the homotopy
  classes of paths in~\(X\) starting at~\(\Cov(y)\) which lift to
  homotopy classes of loops at~\(y\). It is a standard result,
  \cite[Proposition~1.31]{Hatcher2002Alg-Top-Book} , that such
  homotopy classes of paths starting at~\(\Cov(y)\) exactly the
  subgroup~\(\Cov_* (\pi_1(Y,y))\subseteq \FGp(X,\Cov(y))\).
	
  \noindent (3): The action is free \emph{iff} stabiliser at each
  point of~\(Y\) is trivial. Due to~(2) above, this means the action
  is free \emph{iff} \(\Cov_*(\FGp(Y,y))\) is the trivial subgroup
  of~\(\FGp(X,\Cov(y))\) for each~\(y\in Y\). This happen \emph{iff}
  \(Y\) is the universal covering space.  .
\end{proof}

\begin{remark}
  \label{rem:iso-for-FGp-action}
  In Proposition~\ref{prop:nature-of-act-of-FG},
  let~\(\alpha\colon \FGd(X)\times_{s,X, \Cov} Y \to Y\) denote the
  action of~\(\FGd(X)\) on~\(Y\). For \(x\in X\), let \(\alpha_x\) be
  the restriction of this action to~\(\FGp(X,x)\subseteq
  \FGd(X)\). Then note that the proof of~(2) in the proposition also
  implies that the isotropy of~\(\alpha\) and~\(\alpha_x\) are same.
\end{remark}

Last Proposition~\ref{prop:nature-of-act-of-FG} describes covering
spaces as \(\FGd(X)\)\nb-spaces (in which the covering maps are
serving as the momentum maps for the actions). But not every
\(\FGd(X)\)\nb-space can be a covering space as we can easily
construct~\(\FGd(X)\)\nb-spaces in which the momentum maps are not
local homeomorphisms, see
Example~\ref{exmp:justifi-local-homeo}. However, adding the extra
hypothesis that the momentum map of action is {\etale}\footnote{By an
  {\etale} map we mean a local homeomorphism.}, produces the converse
of Proposition~\ref{prop:nature-of-act-of-FG} which is our next main
result Theorem~\ref{thm:classify-cover-space}. In this theorem we also
give other characterisations of~\(\FGd(X)\)\nb-spaces.  Next we
discuss some lemmas required to prove
Theorem~\ref{thm:classify-cover-space}.

\begin{lemma}
  \label{lem:UPL-HL:covering-map}
  Let \(X\) be a locally path connected and semilocally simply
  connected space. Let \(p \colon Y \to X\) be an open surjection;
  assume that~\(p\) has the unique homotopy lifting property. Then
  \(p\) is a covering map.
\end{lemma}
\begin{proof}
  Let \(x\in X\), and let \(U\) be a path connected and relatively
  inessential neighbourhood of \(x\). Let \(\PathS{U}_x\) be the set
  of all path in~\(U\) starting at~\(x\). Now, for given
  \(\tilde{x} \in p^{-1}(x)\), define the set
  \[
    \tilde{U}_{\tilde{x}} \defeq \{\Lift_{\Lift[x]}(1) : \gamma \in
    \PathS{U}_x \}.
  \]
  Since~\(p\) has the unique path lifting property, a standard
  argument shows that for two preimages \(\tilde{x}\neq \tilde{y}\)
  of~\(x\),
  \(\tilde{U}_{\tilde{x}}\cap \tilde{U}_{\tilde{y}}=\emptyset\).

  Next we show that
  \[
    p^{-1} (U) = \bigsqcup_{\Lift[x] \in p^{-1}(x)}
    \Lift[U]_{\Lift[x]}.
  \]
  By definition of \(\tilde{U}_{\tilde{x}}\), it is clear that
  \(\tilde{U}_{\tilde{x}}\subseteq p\inverse(U)\) for each
  \(\tilde{x}\in p\inverse(x)\). Therefore,
  \(\bigsqcup_{\Lift[x] \in p^{-1}(x)} \Lift[U]_{\Lift[x]} \subseteq
  p\inverse(U)\).
	
  For converse, suppose \(y\in p\inverse (U)\).  Let \(\gamma\) be a
  path in~\(U\) connecting~\(x\) to~\(p(y)\); let \(\gamma^-\) be the
  path obtained by traversing~\(\gamma\) in the opposite
  direction. Let \(\tilde{\gamma}^-_y\) be the unique lift
  of~\(\gamma^-\) starting at~\(y\).
  Put~\(\tilde{x} = \tilde{\gamma}^-_y(1)\). Then
  \(\tilde{x}\in p\inverse(x)\) and \(y\in \tilde{U}_{\tilde{x}}\).
	
  We now show that for any \(\Lift[x]\in p^{-1}(x)\),
  \(p|_{\Lift[U]_{\Lift[x]}}\colon \Lift[U]_{\Lift[x]} \to U\) is a
  homeomorphism. Firstly note that the restricted map
  \(p|_{\Lift[U]_{\Lift[x]}}\) is surjective as~\(U\) is path
  connected. The map is injective because~\(U\) is relatively
  inessential and~\(p\) has the unique homotopy lifting
  property. Finally, we prove that~\(p|_{\Lift[U]_{\Lift[x]}}\) is
  open. As \(p\) was an open map, to prove
  that~\(p|_{\Lift[U]_{\Lift[x]}}\) is open it is sufficient to show
  \(\Lift[U]_{\Lift[x]}\) is an open set. This can be proved as
  follows: let \(\Lift_{\Lift[x]}(1) \in \Lift[U]_{\Lift[x]}\) be any
  point. Using the continuity of~\(p\), choose an open set
  \(V\subseteq Y\) containing \(\Lift_{\Lift[x]}(1)\) and with
  \(p(V) \subseteq U\). Then \(V\)~is, in fact, contained
  in~\(\tilde{U}_{\tilde{x}}\). To see this, let \(v\in V\), and
  choose a path \(\xi \in \PathS{U}\) from \(x\) to \(p(v)\). By the
  uniqueness of the path lifting property we have
  \( v = \Lift[\xi]_{\Lift[x]} (1) \in \Lift[U]_{\Lift[x]}\).
\end{proof}

\begin{lemma}
  \label{lem:uniqness-path-lift}
  Let \(p\colon Y \to X\) be a local homeomorphism having the path
  lifting property. Then
  \begin{enumerate}
  \item \(p\) has unique path lifting property;
  \item \(p\) has unique homotopy lifting property.
  \end{enumerate}
\end{lemma}
\begin{proof}
  (1): Let \(\gamma \colon \UInt \to X \) be a path with two lifts
  \(\tilde{\gamma}\) and \(\eta\) starting at
  \(y\in p^{-1}(\gamma(0))\).  Let
  \(A = \{s\in \UInt : \tilde{\gamma}(t) = \eta(t) \textup{ for all }
  t \leq s \}\). Then~\(A\) is a nonempty closed set of the unit
  interval: \(A\neq \emptyset\) for \(0\in A\); and the closedness
  of~\(A\) follows from the continuity of the maps~\(\tilde{\gamma}\)
  and~\(\eta\) and Hausdorffness of~\(Y\).
	
  Now our claim is that \(\sup(A) \defeq s_0 =1\). On the contrary,
  suppose that \(s_0 < 1\). Since~\(A\subseteq \UInt\) is closed,
  \(s_0\in A\), that is, \(\tilde{\gamma}(s_0) = \eta(s_0) =
  y_0\). Choose a neighbourhood~\(U\) of~\(y_0\) such that \(p|_U\) is
  homeomorphism onto its image and \(p(U)\subseteq X\) is open. Note
  that \(\gamma = p\circ \tilde{\gamma} = p\circ \eta\). The
  continuity of~\(\gamma\) at~\(s_0\) gives us \(\epsilon>0\) such
  that \(\gamma((s_0-\epsilon, s_0+\epsilon)) \subseteq
  p(U)\). As~\(p|_U\colon U\to p(U)\) is homeomorphism,
  \(\tilde{\gamma}(s_0+\epsilon/2) = \eta(s_0+\epsilon/2)\) which
  contradicts that \(s_0= \sup(A)\).
	
  \noindent (2): Last proved claim of~(1) and the fact that~\(p\) is a
  local homeomorphism satisfy hypothesis of~Proposition~3 of
  Chapter--5-6A~\cite{DoCarmo1976Diff-Geo-Curves-Surface}; this
  proposition immediately implies the desired result.
\end{proof}

\begin{lemma}
  \label{lem:moment-UPLP-HLP}
  Let \(Y\) be a left \(\FGd(X)\)\nb-space, where \(X\) is locally
  path connected and semilocally simply connected. Suppose the
  momentum map \(r_Y\colon Y \to X\) is a local homeomorphism.  Then
  the momentum map \(r_Y\) has unique path lifting property and unique
  homotopy lifting property.
\end{lemma}
\begin{proof}
  Given a path~\(\gamma\) in~\(X\) and \(y\in r_Y\inverse(\gamma(0))\)
  define the path~\(\tilde{\gamma}\) in~\(Y\) starting at~\(y\) as
  follows:
  \begin{equation}\label{equ:uniq-lift}
    \Lift(t)= [\gamma|_{[0,t]}]y \quad \text{ for } 0\leq t\leq 1.
  \end{equation}
  The continuity of \(\Lift\) is follows from the continuity of the
  action. Furthermore,
  \[
    r_Y\circ \Lift (t) = r_Y([\gamma|_{[0,t]}]y) =
    r([\gamma|_{[0,t]}]) = \gamma(t)
  \]
  where~\(r\) is the range map of~\({\FGd(X)}\).  Thus~\(\Lift\) is a
  lift of~\(\gamma\) at~\(y\) in~\(Y\). This shows that~\(r_Y\) has
  path lifting property.  Now Lemma~\ref{lem:uniqness-path-lift}
  implies the required claim.
\end{proof}

\begin{theorem}
  \label{thm:classify-cover-space}
  Let \(X\) be a locally path connected and semilocally simply
  connected space.  Suppose \(p\colon Y\to X\) is a (surjective) local
  homeomorphism. Then the following statements are equivalent:
  \begin{enumerate}
  \item \(Y\) is a \(\FGd(X)\)\nb-space;
  \item \(p\) has unique path lifting property;
  \item \(p\) is a covering map;
  \item \(p\) has unique homotopy lifting property.
  \end{enumerate}
	
\end{theorem}
\begin{proof}
  \noindent (1) \(\implies\) \noindent (2) or (4): Follows from
  Lemma~\ref{lem:moment-UPLP-HLP}.
	
  \noindent (2) \(\implies\) \noindent (3): Since~\(p\) is a local
  homeomorphism with unique path lifting property,
  Lemma~\ref{lem:moment-UPLP-HLP}(2) says that~\(p\) has unique
  homotopy lifting property. Now Lemma~\ref{lem:UPL-HL:covering-map}
  shows that~\(p\) is a covering map.
	
  \noindent (3) \(\implies\) \noindent (1): Follows from the first
  part of Proposition~\ref{prop:nature-of-act-of-FG}.
	
  \noindent Finally, (4)\(\implies\)(2) is obvious.
	
\end{proof}

Next examples describes a ~\(\FGd(X)\)\nb-space in which the momentum
map is not a local homeomorphism.
\begin{example}
  \label{exmp:justifi-local-homeo}
  Consider the map \(p\colon \R \times \R \to S^1\) by
  \(p(x,y) = \e^{2\pi ix}\); this map is not a local homeomorphism.
  Equipe~\(\R \times \R\) with the translation (in both variables)
  action of~\(\R\); equipe the unit circle~\(S^1\) with the
  next~\(\R\)\nb-action: \(t\cdot \e^{2\pi ix} = \e^{2\pi i(t+x)}\)
  where \(t, \xi\in \R\). Then \(p\) is an~\(\R\)\nb-equivariant map.
  Now Lemma~\ref{lem:homo-of-trans-gpds} implies that~\(\R\times \R\)
  carries an action of \(\R\ltimes S^1\) with~\(p\) as the momentum
  map.  We identify \(\FGd(S^1) \cong \R \ltimes S^1\) using
  Example~\ref{exa:fgd-of-gp}. Thus \(\R \times \R\) is a
  \(\FGd(S^1)\)\nb-space, but \(p\) is not a covering map.
\end{example}

This point on, we shall restrict our study to the category of
\(\FGd(X)\)\nb-spaces having the momentum map a local
homeomorphism. Our next quests are to characterise free---and, then,
proper---\(\FGd(X)\)\nb-actions on such spaces. The next result gives
us a necessary and sufficient condition for freeness of such actions.

\begin{proposition}
  \label{prop:condi-for-free-action}
  Suppose \(X\) is a locally path connected and semilocally simply
  connected space and~\(Y\) a path connected
  \(\FGd(X)\)\nb-space. Suppose the momentum map \(r_Y\colon Y\to X\)
  is a local homeomorphism. Then the action of \(\FGd(X)\) on \(Y\) is
  free iff \(Y\) is simply connected.
\end{proposition}

\begin{proof}
  Recall from Theorem~\ref{thm:classify-cover-space} that \(r_Y\) is a
  covering map. Now the given action is free \emph{iff} the stabiliser
  of any given point \(y\in Y\) is trivial. Recall from
  Proposition~\ref{prop:nature-of-act-of-FG}(\ref{cond:stab}), that
  the stabiliser of~\(y\) is the
  subgroup~\({r_Y}_* (\pi_1(Y,y))\subseteq \FGd(X)\). This subgroup is
  trivial \emph{iff} \(r_Y\colon Y \to X\) is the universal covering
  space.
\end{proof}

Theorem~\ref{thm:classify-cover-space} suggests that covering space
theory may be rephrased in terms of~\(\FGd(X)\)\nb-spaces. For this,
consider a path connected, locally path connected and semilocally
simply connected space~\(X\). Let \(\CatAct_{\FGd(X)}\) denote the
category of \(\FGd(X)\)\nb-space---the objects of this category are
\(\FGd(X)\)\nb-spaces, and \(\FGd(X)\)\nb-equivariant maps are arrows
between objects. Consider the subcategory~\(\CatEAct_{\FGd(X)}\)
of~\(\CatAct_{\FGd(X)}\) consisting of path connected
\(\FGd(X)\)\nb-spaces whose momentum maps are local homeomorphisms. On
the other hand, let \(\CatCovSp_X\) denote the category of covering
\emph{maps}\footnote{We assume that the corresponding covering spaces
  are path connected.} of~\(X\) that Spanier defines~\cite[Chapter 2,
\S5]{Spanier1966Alg-Top-Book}---the objects in this category are
covering maps and arrows are the continuous maps of covering spaces
which preserve that covering maps. Then
Theorem~\ref{thm:classify-cover-space} establishes an isomorphism of
categories~\(\CatCovSp_X \iso \CatEAct_{\FGd(X)}\): (1) and (2) of
this theorem clearly establish the isomorphism of objects. To show
that arrows are also well behaved, firstly, take two covering spaces
\(Y_1\xrightarrow{p_1} X \xleftarrow{p_2} Y_2\) and consider a
morphism \(f\colon Y_1 \to Y_2\). Then as a consequence of
Remark~\ref{rem:img-of-lift-of-path},~\(f\) is
\(\FGd(X)\)\nb-equivariant map. Conversely, given a
\(\FGd(X)\)\nb-equivariant map \(g\colon Y_1 \to Y_2\), by definition
of equivariant map \(r_{Y_2} \circ g = r_{Y_1}\). That means~\(g\) is
a morphism of covering space \(g\colon (Y_1,p_1) \to (Y_2,p_2)\).  We
summarise this discussion as the next theorem:
\begin{theorem}
  \label{thm:iso-act-and-cov-categiries}
  Let \(X\) be a path connected, locally path connected and
  semilocally simply connected space. Then the
  categories~\(\CatEAct_{\FGd(X)}\) and~\(\CatCovSp_X\) are
  isomorphic.
\end{theorem}

Furthermore, Proposition~\ref{prop:condi-for-free-action} identifies
the \emph{universal} covering space with a free \(\FGd(X)\)\nb-space
in~\(\CatEAct_{\FGd(X)}\). Therefore, up to equivariant homeomorphism,
there is a \emph{unique path connected free \(\FGd(X)\)\nb-space}
having the momentum map a local homeomorphism. In fact, other
\(\FGd(X)\)\nb-spaces are quotients this free space; using this
observation, proposing a universal property for the free
\(\FGd(X)\)\nb-space which makes~\(\FGd(X)\) the \emph{universal
  \(\FGd(X)\)\nb-space}, should a good exercise.

Note that other \(\FGd(X)\) spaces are quotients of the universal
covering spaces. Thus the universal covering space \emph{seems} the
initial object of~\(\mathcal{E}_{\FGd(X)}\), unlike the classifying
space of proper \(G\)\nb-actions in~\cite[Definition
1.6]{Baum-Connes-Higson1993Classyfying-spaces-for-prop-act-and-K-theory}
which is the \emph{terminal} object in appropriate sense.

Rephrasing the standard results about covering spaces using the
identification~\(\CatEAct_{\FGd(X)}\iso \CatCovSp_X\) can be an
interesting exercise. Next are two examples of it:

\begin{proposition}[Consequence of Lemma~80.2 in
  Munkres~\cite{Munkress1975Topology-book} and Theorem~\ref{thm:iso-act-and-cov-categiries}]
  Let \(Y\) and \(Z\) be \(\FGd(X)\)\nb-spaces and \(\omega: Y \to Z\)
  a map of spaces. Next two statements are equivalent:
  \begin{enumerate}
  \item \(\omega\) is a \(\FGd(X)\)\nb-equivariant map.
  \item \(\omega \circ r_Z = r_Y\).
  \end{enumerate}
  Moreover, if any one of above holds, then following hold:
  \begin{enumerate}[resume]
  \item \(\omega\) is a covering map.
  \item \(Y\) is a \(\FGd(Z)\)\nb-space with \(\omega\) as the
    momentum map (and the action is given by evaluation of lifted path
    homotopies at~\(1\)).
  \end{enumerate}
\end{proposition}

\begin{proposition}[Theorem~80.1 in
  Munkres~\cite{Munkress1975Topology-book} stated using
  Theorem~\ref{thm:iso-act-and-cov-categiries}]
  Let \(Y\) be a \(\FGd(X)\)\nb-space and \(y_0\in Y\). Let
  \(H\subseteq \FGd(X)\) be the stabliser at~\(y_0\). Then the group
  of \(\FGd(X)\)\nb-equivariant homeomorphisms of~\(Y\) is isomorphic
  to~\(N(H)/H\) where \(N(H)\) is the normaliser of~\(H\)
  in~\(\FGp(X,p(y_0))\).
\end{proposition}

\noindent The last proposition uses,
Proposition~\ref{prop:nature-of-act-of-FG}(2), namely, the
isotropy~\( H = {r_Y}_*(\FGp(Y, y_0))\subseteq \FGp(X, r_Y(y_0))\).

\section{Proper actions}
\label{sec:proper-actions}

\subsection{The kinetics of action}
\label{sec:kinetics-action}

Let \(A\xrightarrow{f} X \xleftarrow{g} B\) be maps of spaces. For
\(P\subseteq A\) and \(Q\subseteq B\), we denote the subset
\((P\times Q) \cap (A\times_{f,X,g}B)\) of the fibre
product~\( A\times_{f,X,g}B\) by~\(P\times_{f,X,q} Q\). The
set~\(P\times_{f,X,q} Q\) can be empty.

Before moving onto the proper actions of \(\FGd(X)\), we discuss a
technical property of the kinetics of \(\FGd(X)\)\nb-action, namely,
Lemma~\ref{lem:action-map-local-homeo}(3) and
Proposition~\ref{prop:kinetics-is-covering}.  Fix a
path connected covering space~\(p\colon Y\to X\), equivalently, a path
connected \(\FGd(X)\)\nb-space in which the momentum map is a local
homeomorphism. Recall from Equation~\eqref{eq:kine} that the map of
kinetics of the action is given by
\[
  a\colon \FGd(X) \times_{s, X, p} Y \to Y\times Y, \quad a([\gamma],
  y) = (\widetilde{\gamma_y}(1), y)
\]
where \(([\gamma], y)\in \FGd(X) \times_{s, X, p} Y\), and
\(\widetilde{\gamma_y}\) is the unique lift of~\(\gamma\) in~\(Y\)
starting at~\(y\). In other words,
\( a([\gamma], y) = (\widetilde{\gamma_y}(1),
\widetilde{\gamma_y}(0))\).

Since the action of~\(\FGd(X)\) on~\(Y\) is continuous,~\(a\) is
continuous. As~\(Y\) is path connected,~\(a\) is
surjective. Lemma~\ref{lem:free-act} and
Proposition~\ref{prop:condi-for-free-action} imply that the
kinetics~\(a\) is one-to-one \emph{iff} \(Y\) is the simply connected
covering space of~\(X\). In what follows, we show that~\(a\) is a
covering map, and we shall describe slices of~\(a\) in
Proposition~\ref{prop:kinetics-is-covering}. This technical
observation shall prove useful to describe proper actions
of~\(\FGd(X)\).

Consider the following collection~\(\mathcal{L}\) of open subsets
of~\(\FGd(X) \times_{s, X, p} Y\): an element in~\(\mathcal{L}\) is of
the form \(N([\gamma], p(U), p(V)) \times_{s, X, p} V\) where
\begin{itemize}
\item \([\gamma]\in \FGd(X)\);
\item \(U, V\subseteq Y\) are path connected and relatively
  inessential slices of~\(p\) such that \(\gamma(1)\in p(U)\) and
  \(\gamma(0)\in p(V)\).
\end{itemize}
Since path connected and relatively inessential slices of~\(p\) form a
basis of~\(Y\), \(\mathcal{L}\) is a basis
of~\(\FGd(X) \times_{s, X, p} Y\). Moreover, if \(U\) and \(V\) are nonempty,
then so is~\(N([\gamma], p(U), p(V)) \times_{s, X, p} V\).
  
  \begin{lemma}\label{lem:action-map-local-homeo}
    Let \(U,V\subseteq Y\) be nonempty path connected relatively
    inessential open slices.
    \begin{enumerate}
    \item The map of kinetics~`\(a\)' maps the basic open set
      \(N([\gamma], p(U), p(V)) \times_{s, X, p} V\) bijectively onto
      the basic open set~\(U\times V\) of~\(Y\times Y\).
    \item \(a\) is an open map.
    \item \(a\) is a local homeomorphism. In particular, restriction
      of~\(a\) to the basic open set
      \(N([\gamma], p(U), p(V)) \times_{s, X, p} V\) is a
      homeomorphism onto the basic open set~\(U\times V\).
    \end{enumerate}
  \end{lemma}
  
  \begin{proof}
    (1) Write \(B\defeq N([\gamma], p(U), p(V)) \times_{s, X, p} V
    \). Since~\(U\) and~\(V\) are path connected and relatively
    inessential slices, \(a|_B\colon B\to U\times V\) is a continuous
    bijection.

    \noindent (2): The last argument shows that~\(a\) maps a basic
    open set in~\(\FGd(X)\times_{s,X,p} Y\) to a basic open set
    in~\(Y\times Y\). Therefore, \(a\) is an open mapping.

    \noindent (3): Follows from (1) and (2) above.
  \end{proof}

  Fix two path connected relatively inessential
  slices~\(U,V\subseteq Y\).  We next want to describe
  \(a\inverse(U\times V)\). For that, fix points \(y\in U\) and
  \(z\in V\). Denote the transformation groupoid \(\FGd(X)\ltimes Y\)
  by~\(A\). Consider the set~\(A_z^y\) of arrows in~\(A\) which
  \emph{take} the unit~\(z\) in~\(A\) to~\(y\) for the obvious action
  of~\(A\) on~\(\base[A]\homeo Y\). To be precise,
  \[
    A_z^y \defeq \{([\gamma], y) \in A : \widetilde{\gamma_z}(1) =
    y\}.
  \]
  In other words, \(A_z^y\) consists of arrows in~\(\FGd(X)\) which
  take~\(z\in Y\) to~\(y\) under the action of the fundamental
  groupoid. Therefore, we may also write
  \begin{equation}\label{eq:alt-def-of-A-zy}
    A_z^y = \{[\gamma] \in \FGd(X) : \widetilde{\gamma_z}(1) =
    y\};
  \end{equation}
  this identification is more comfortable to use than the earlier one.
  
  \begin{lemma}
    \label{lem:fibres-over-V-V}
    Let \([\gamma_1],[\gamma_2]\in A_z^y\); define
    \(B_i = N([\gamma_i], p(U), p(V)) \times_{s,X,p} V \) for
    \(i=1,2\). Then
    \[
      B_1\cap B_2 =
      \begin{cases}
        \emptyset &  \text{ if } [\gamma_1]\neq [\gamma_2],\\
        B_1 & \text{ if } [\gamma_1]= [\gamma_2].
      \end{cases}
    \]
  \end{lemma}
  \begin{proof}
    Assume \(B_1\cap B_2\neq \emptyset\), and let \([\eta]\) be in the
    intersection. Then
    \( [\eta] = [\delta_2 \oblong \gamma_2 \oblong \delta_1] =
    [\epsilon_2 \oblong \gamma_1 \oblong \epsilon_1] \) for paths
    \(\delta_2,\epsilon_2\) laying in~\(p(U)\) starting at~\(p(y)\),
    and paths \(\delta_1,\epsilon_1\) laying in~\(p(V)\) ending
    at~\(p(z)\). Therefore,
    \[
      [\gamma_2 ] = [\delta_2^- \oblong \epsilon_2] \oblong [\gamma_1]
      \oblong [\epsilon_1 \oblong \delta_1^-]
    \]
    where \(\delta_i^-\) the reverse of~\(\delta_i\) for
    \(i=1,2\). Now, since \(p(U)\) is relatively inessential, the
    \emph{loop}~\(\delta_2^- \oblong \epsilon_2\) at~\(y\) is null
    homotopic. So is~\(\epsilon_1 \oblong \delta_1^-\). Therefore the
    last equation implies that~\([\gamma_1] = [\gamma_2]\). Thus
    \([\gamma_1]\neq [\gamma_2]\) gives \(B_1\cap B_2=\emptyset\). The
    other case is clear.
  \end{proof}

  As a consequence of last lemma, using
  Equation~\eqref{eq:alt-def-of-A-zy}, we can see that
  \begin{equation*}\label{eq:a-inverse-U-V}
    a\inverse(U\times V) = \bigsqcup_{[\gamma] \in A_z^y}
    \left(N([\gamma], p(U), p(V)) \times_{s, X, p}
      V \right).
  \end{equation*}
  Moreover, Lemma~\ref{lem:action-map-local-homeo} shows that
  restriction of the kinetics to
  each~\(N([\gamma], p(U), p(V)) \times_{s, X, p} V\) above is a
  homeomorphism onto~\(U\times V\). Thus, we have prove the following
  results:

  \begin{proposition}\label{prop:kinetics-is-covering}
    Let~\(X\) be path connected, locally path connected and
    semilocally simply connected space. Let \(Y\) be a path connected
    \(\FGd(X)\)\nb-space having \etale\ momentum map \(r_Y\). Then the
    map
    \[
      a\colon \FGd(X)\times_{s,X, r_Y} Y \to Y\times Y
    \]
    of the kinetics of the action is a covering map. In fact, for path
    connected relatively inessential slices~\(U,V\subseteq Y\),
    \[
      a\inverse(U\times V) = \bigsqcup_{[\gamma] \in A_z^y}
      \left(N([\gamma], p(U), p(V)) \times_{s, X, p} V \right)
    \]
    where each \(N([\gamma], p(U), p(V)) \times_{s, X, p} V\) is a
    slice over~\(U\times V\) under~\(a\).
  \end{proposition}

  Reader may compare Proposition~\ref{prop:kinetics-is-covering}
  and~\cite[Observation 2.10]{Holkar-Hossain2023Top-FGd-I}; in the
  latter one, one should consider the action of~\(\FGd(X)\)
  on~\(X\). Last proposition
  generalises~\cite[Proposition~2.37]{Reinhart1983Folliations-book}.

  Using Proposition~\ref{prop:kinetics-is-covering}, we can describe
  \emph{small} compact sets in the transformation groupoid~\(A\) as
  follows. Assume~\(Y\) is locally compact, and consider the same open
  sets~\(U\) and~\(V\) as in last proposition.  Let
  \(U',V' \subseteq Y\) be path connected relatively inessential
  slices whose closures are compact and \(\overline{U'} \subseteq U\)
  and~\(\overline{V'} \subseteq V\).  Then for \([\gamma] \in A_z^y\),
  the closure
  \(\overline{N([\gamma], p(U'), p(V')) \times_{s,X,p} V'}\) is
  homeomorphic to the compact
  subset~\(\overline{U'}\times \overline{V'} = \overline{U' \times
    V'}\subseteq U\times V\). Then the proposition implies that
  \begin{equation}
    \label{eq:inverse-im-of-small-cpt}
    a\inverse(\overline{U'\times V'}) = \bigsqcup_{[\gamma] \in A_z^y}
    \left( \overline {N([\gamma], p(U'), p(V')) \times_{s, X, p} V'} \right)  .
  \end{equation}

  As a closing remark, we note that the collection~\(\mathcal{L}'\)
  consisting of path connected relatively, inessential open sets~\(U'\)
  such that \(U'\) closure is compact and the closure is contained in
  a path connected, relatively inessential slice forms a basis for the
  topology on~\(Y\) when~\(Y\) is locally compact, Hausdorff, path
  connected, locally path connected and semilocally simply
  connected. Moreover, \(\mathcal{L}'\) is a refinement
  of~\(\mathcal{L}\).

  \subsection{Proper actions of the fundamental groupoid}

  In this section, we study proper actions of the locally compact
  groupoid~\(\FGd(X)\). We prove
  Theorem~\ref{thm:suff-cond-proper-action} which characterises proper
  actions using isotropy at a point. Indeed, we focus on the
  \(\FGd(X)\)\nb-spaces in the category~\(\CatEAct_X\).

  Suppose a path connected, locally path connected and semilocally
  simply connected space~\(X\) is given. Let \(p\colon Y\to X\) be a
  path connected covering space. If~\(X\) is locally compact and
  Hausdorff, then the transformation
  groupoid~\(A\defeq \FGd(X) \ltimes Y\) is locally compact
  Hausdorff. This can be seen as follows: \cite[Section
  3]{Holkar-Hossain2023Top-FGd-I} implies that the simply connected
  covering space~\(\tilde{X}\) is Hausdorff and locally compact. In
  this case, the deck transformation action on~\(\tilde{X}\) is
  proper, see~\cite[Chapter~21, \S Covering
  manifold]{Lee2012Intro-to-smooth-manifolds-book}. As a consequence,
  \(Y\)---which is quotient of~\(X\) by a subgroup of the deck
  transformation---is also locally compact and Hausdorff. Next
  \(\FGd(X)\) is locally compact and Hausdorff if~\(X\) is
  so. Therefore, the transformation groupoid \(\FGd(X)\ltimes Y\) is
  locally compact and Hausdorff.

  \begin{theorem}
    \label{thm:suff-cond-proper-action}
    Let \(X\) be a locally compact, Hausdorff, path connected, locally
    path connected and semilocally simply connected space.  Suppose
    \(Y\) is a path connected \(\FGd(X)\)\nb-space with momentum map
    \(p\colon Y\to X\) a surjective local homeomorphism. Then
    following are equivalent:
    \begin{enumerate}
    \item The action of~\(\FGd(X)\) on~\(Y\) is proper.
    \item The fundamental group of~\(Y\) is finite.
    \item \(p_*(\FGp(Y))\) is a finite subgroup of~\(\FGp(X)\).
    \item The restricted action of the fundamental group~\(\FGp(X)\)
      on~\(Y\) is proper.
    \end{enumerate}
  \end{theorem}
  \begin{proof}
  
    (2)\(\iff\)(3): Theorem~\ref{thm:classify-cover-space} implies
    that \(Y\) is a covering space of~\(X\) with~\(p\) are the
    covering map. Therefore, (2) and (3) are clearly equivalent as the
    group homomorphism \(p_*\colon \FGp(Y) \to\FGp(X)\) is injective.

    \noindent(1)\(\implies\)(2):
    Observation~\ref{obs:isotropy-of-proper-action} says that for
    proper action the isotropy is a compact
    set. Proposition~\ref{prop:nature-of-act-of-FG}(2) says that the
    isotropy in current case is the fundamental group~\(\FGp(Y)\)
    which is compact \emph{iff} it is finite.

    \noindent (3)\(\implies\)(1): This is the longest part of the
    proof and is be done in the end. Module this proof, the first
    three are equivalent.

    \noindent (1)\(\implies\)(4): For
    any~\(x\in X\),\(\FGp(X,x)\subseteq \FGd(X)\) is a closed
    subgroup. Therefore, if the action of~\(\FGd(X)\) is proper, the
    restriction of the action to~\(\FGp(X,x)\) is also proper.

    \noindent (4)\(\implies\)(2): Remark~\ref{rem:iso-for-FGp-action}
    identifies the isotropy of the restricted action~\(\FGp(X)\)
    with~\(\FGp(Y)\iso p_*(\FGp(Y))\). Therefore, the restricted
    action is proper implies the isotropy is finite.

    Finally, we prove only unjustified claim (3)\(\implies\)(1) which
    shall complete the proof. The claim is proved in three steps:
    \begin{enumerate}
    \item Firstly, when \(K\subseteq Y\times Y\) is singleton.
    \item Then, when \(K\) is a small compact set as in
      Equation~\eqref{eq:inverse-im-of-small-cpt}.
    \item Finally, for a general compact set~\(K\).
    \end{enumerate}
  
    Before we start, let \(A\) denote the transformation
    groupoid~\(\FGd(X)\ltimes Y\). And note that the isotropy
    at~\(z\in Y\) for the action of~\(\FGd(X)\) is the stabiliser
    subgroup~\(A_z^z\subseteq A\) which is assumed to be
    finite. Recall Equation~\eqref{eq:alt-def-of-A-zy}; and write an
    enumeration of \(A_z^z = \{[\gamma_1], \dots, [\gamma_n]\}\) for
    some \(n\in \N\).

    Then, in the first case, when \(K=\{(y,z)\}\),
    \(a\inverse(\{(y,z)\}) = A_z^y\).
    Observation~\ref{obs:cardinality-of-iso-and-edges} implies that
    \(A_z^y\) is in bijection with~\(A_z^z\) hence it is compact.

    Now consider the basis~\(\mathcal{L}'\) for the topology of~\(X\)
    discussed just after
    Equation~\eqref{eq:inverse-im-of-small-cpt}. This basis consisting
    of relatively compact open sets~\(U'\) whose closures are
    contained in a path connected relatively inessential slice
    over~\(p\colon Y\to X\). Then the sets of the form \(U'\times V'\)
    where \(U',V'\in \mathcal{L}'\) form a basis of relatively compact
    sets for the topology of~\(Y\times Y\). For
    \(U',V'\in \mathcal{L}'\),
    Equation~\eqref{eq:inverse-im-of-small-cpt} implies that
    \[
      a\inverse(\overline{U'\times V'}) = \bigsqcup_{i=1}^n
      \left(\overline {N([\gamma_i], p(U'), p(V')) \times_{s, X, p}
          V'} \right)
    \]
    where each
    \(\overline {N([\gamma_i], p(U'), p(V')) \times_{s, X, p} V'}\) is
    a homeomorphic copy of~\(\overline{U'\times V'}\) (
    see~Proposition~\ref{prop:kinetics-is-covering}). Thus
    \(a\inverse(\overline{U'\times V'})\) is a compact set being union
    of finitely many compact sets.

    Finally, let \(K\subseteq Y\times Y\) be any compact
    set. Cover~\(K\) by finitely many relatively compact
    sets~\(U_1\times V_1, \dots, U_m\times V_m\) where
    \(U_j,V_j\in \mathcal{L}'\). Then
    \[
      a\inverse(K)\subseteq \bigcup_{i=1}^m
      a\inverse(\overline{U_j\times V_j})
    \]
    where each \(a\inverse(\overline{U_j\times V_j})\) is compact by
    last argument. Thus \(a\inverse(K)\subseteq A\) is compact.
  \end{proof}

  Following are some immediate consequences of
  Theorem~\ref{thm:suff-cond-proper-action}(2).

  \begin{corollaries}
    \label{cor:cons-prop-act-theorem}
    Let \(X\) be a path connected, locally path connected, semilocally
    simply connected, Hausdorff and locally compact space. Then
    following hold:
    \begin{enumerate}
    \item The action of~\(\FGd(X)\) on the simply connected covering
      space is proper.
    \item The action of~\(\FGd(X)\) on~\(X\) is proper \emph{iff} the
      fundamental group of~\(X\) is finite.
    \item For each finite subgroup~\(G\) of~\(\FGp(X)\), the
      associated covering space~\(X_G\to X\) is a proper covering
      space.  Moreover, these are the only proper
      \(\FGd(X)\)\nb-spaces in the action
      category~\(\CatEAct_{\FGd(X)}\).
    \end{enumerate}
  \end{corollaries}

\medskip

\paragraph{\itshape Acknowledgement:} We are grateful to Prahlad
Vaidyanathan, Angshuman Bhattacharya and Jean Renault for helpful
discussions. We thank Suliman Albandik for his patience and beneficial
discussions. This work was supported by SERB's SRG/2020/001823 grant
of the first author and CSIR's 09/1020(0159)/2019-EMR-I grant of the
second author; we thank the funding institutions.


\begin{bibdiv}
\begin{biblist}

\bib{Baum-Connes-Higson1993Classyfying-spaces-for-prop-act-and-K-theory}{incollection}{
      author={Baum, Paul},
      author={Connes, Alain},
      author={Higson, Nigel},
       title={Classifying space for proper actions and {$K$}-theory of group
  {$C^\ast$}-algebras},
        date={1994},
   booktitle={{$C^\ast$}-algebras: 1943--1993 ({S}an {A}ntonio, {TX}, 1993)},
      series={Contemp. Math.},
      volume={167},
   publisher={Amer. Math. Soc., Providence, RI},
       pages={240\ndash 291},
         url={http://dx.doi.org/10.1090/conm/167/1292018},
      review={\MR{1292018}},
}

\bib{Brown-Danesh-1975-Top-FG-1}{article}{
      author={Brown, R.},
      author={Danesh-Naruie, G.},
       title={The fundamental groupoid as a topological groupoid},
        date={1974/75},
        ISSN={0013-0915},
     journal={Proc. Edinburgh Math. Soc. (2)},
      volume={19},
       pages={237\ndash 244},
         url={http://dx.doi.org/10.1017/S0013091500015509},
      review={\MR{0413096}},
}

\bib{Brown-Danesh-Naruie-1976-Top-FG-2}{article}{
      author={Brown, R.},
      author={Danesh-Naruie, G.},
      author={Hardy, J. P.~L.},
       title={Topological groupoids. {II}. {C}overing morphisms and
  {$G$}-spaces},
        date={1976},
        ISSN={0025-584X},
     journal={Math. Nachr.},
      volume={74},
       pages={143\ndash 156},
         url={http://dx.doi.org/10.1002/mana.3210740110},
      review={\MR{0442937}},
}

\bib{Brown2006Topology-book}{book}{
      author={Brown, Ronald},
       title={Topology and groupoids},
   publisher={BookSurge, LLC, Charleston, SC},
        date={2006},
        ISBN={1-4196-2722-8},
        note={Third edition of {{\i}t Elements of modern topology}
  [McGraw-Hill, New York, 1968; MR0227979], With 1 CD-ROM (Windows, Macintosh
  and UNIX)},
      review={\MR{2273730}},
}

\bib{DoCarmo1976Diff-Geo-Curves-Surface}{book}{
      author={do~Carmo, Manfredo~P.},
       title={Differential geometry of curves and surfaces},
   publisher={Prentice-Hall, Inc., Englewood Cliffs, N.J.},
        date={1976},
        note={Translated from the Portuguese},
      review={\MR{0394451}},
}

\bib{Emerson-Meyer2010Dualities-in-Equi-Kasparov-theory}{article}{
      author={Emerson, Heath},
      author={Meyer, Ralf},
       title={Dualities in equivariant {K}asparov theory},
        date={2010},
     journal={New York J. Math.},
      volume={16},
       pages={245\ndash 313},
         url={http://nyjm.albany.edu:8000/j/2010/16_245.html},
      review={\MR{2740579}},
}

\bib{Hatcher2002Alg-Top-Book}{book}{
      author={Hatcher, Allen},
       title={Algebraic topology},
   publisher={Cambridge University Press, Cambridge},
        date={2002},
        ISBN={0-521-79160-X; 0-521-79540-0},
      review={\MR{1867354}},
}

\bib{Holkar2017Construction-of-Corr}{article}{
      author={Holkar, Rohit~Dilip},
       title={Topological construction of {$C^*$}-correspondences for groupoid
  {$C^*$}-algebras},
        date={2017},
     journal={Journal of {O}perator {T}heory},
      volume={77:1},
      number={23-24},
       pages={217\ndash 241},
}

\bib{Holkar-Hossain2023Top-FGd-I}{article}{
      author={Holkar, Rohit~Dilip},
      author={Hossain, Md~Amir},
       title={Topological fundamental groupoid.{I}.},
        date={2023},
        note={Preprint- arXiv:2302.01583v1},
}

\bib{Lee2012Intro-to-smooth-manifolds-book}{book}{
      author={Lee, John~M.},
       title={Introduction to smooth manifolds},
     edition={Second},
      series={Graduate Texts in Mathematics},
   publisher={Springer, New York},
        date={2013},
      volume={218},
        ISBN={978-1-4419-9981-8},
      review={\MR{2954043}},
}

\bib{Muhly-Tomforde-2005-Topological-quivers}{article}{
      author={Muhly, Paul~S.},
      author={Tomforde, Mark},
       title={Topological quivers},
        date={2005},
        ISSN={0129-167X},
     journal={Internat. J. Math.},
      volume={16},
      number={7},
       pages={693\ndash 755},
         url={http://dx.doi.org/10.1142/S0129167X05003077},
      review={\MR{2158956 (2006i:46099)}},
}

\bib{Munkress1975Topology-book}{book}{
      author={Munkres, James~R.},
       title={Topology: a first course},
   publisher={Prentice-Hall, Inc., Englewood Cliffs, N.J.},
        date={1975},
      review={\MR{0464128 (57 \#4063)}},
}

\bib{Reinhart1983Folliations-book}{book}{
      author={Reinhart, Bruce~L.},
       title={Differential geometry of foliations},
      series={Ergebnisse der Mathematik und ihrer Grenzgebiete [Results in
  Mathematics and Related Areas]},
   publisher={Springer-Verlag, Berlin},
        date={1983},
      volume={99},
        ISBN={3-540-12269-9},
         url={https://doi.org/10.1007/978-3-642-69015-0},
        note={The fundamental integrability problem},
      review={\MR{705126}},
}

\bib{Renault1980Gpd-Cst-Alg}{book}{
      author={Renault, Jean},
       title={A groupoid approach to {$C^{\ast} $}-algebras},
      series={Lecture Notes in Mathematics},
   publisher={Springer, Berlin},
        date={1980},
      volume={793},
        ISBN={3-540-09977-8},
      review={\MR{584266 (82h:46075)}},
}

\bib{Spanier1966Alg-Top-Book}{book}{
      author={Spanier, Edwin~H.},
       title={Algebraic topology},
   publisher={McGraw-Hill Book Co., New York-Toronto, Ont.-London},
        date={1966},
      review={\MR{0210112}},
}

\bib{Tu2004NonHausdorff-gpd-proper-actions-and-K}{article}{
      author={Tu, Jean-Louis},
       title={Non-{H}ausdorff groupoids, proper actions and {$K$}-theory},
        date={2004},
        ISSN={1431-0635},
     journal={Doc. Math.},
      volume={9},
       pages={565\ndash 597 (electronic)},
      review={\MR{2117427 (2005h:22004)}},
}

\bib{Williams2019A-Toolkit-Gpd-algebra}{book}{
      author={Williams, Dana~P.},
       title={A tool kit for groupoid {$C^*$}-algebras},
      series={Mathematical Surveys and Monographs},
   publisher={American Mathematical Society, Providence, RI},
        date={2019},
      volume={241},
        ISBN={978-1-4704-5133-2},
         url={https://doi.org/10.1016/j.physletb.2019.06.021},
      review={\MR{3969970}},
}

\end{biblist}
\end{bibdiv}
\end{document}